\documentclass{tran-l}

\newtheorem{theorem}{Theorem}[section]
\newtheorem{lemma}[theorem]{Lemma}

\theoremstyle{definition}
\newtheorem{definition}[theorem]{Definition}
\theoremstyle{definition}
\newtheorem{proposition}[theorem]{Proposition}
\newtheorem{example}[theorem]{Example}

\theoremstyle{remark}
\newtheorem{remark}[theorem]{Remark}

\numberwithin{equation}{section}

\begin{document}

\title{On Wishart and non-central Wishart
distributions on symmetric cones}


%    author one information
\author{Eberhard Mayerhofer}
\address{Department of Mathematics and Statistics, University of Limerick, Castletroy, Ireland}
%\curraddr{}
\email{eberhard.mayerhofer@ul.ie}
%\thanks{}

%    \subjclass is required.
%
%\begin{keyword}
%\kwd{}
%\end{keyword}

\date{}

%\dedicatory{Gerard Letac} 

%    Abstract is required.
\begin{abstract}

Necessary conditions for the existence of non-central Wishart distributions are given. Our method relies on positivity properties of spherical polynomials on Euclidean Jordan Algebras and advances an approach by Peddada and Richards (1991), where only a special case (positive semidefinite matrices, rank one non-centrality parameter) is treated. Not only needs the shape parameters be in the Wallach set - as is the case for Riesz measures - but also the rank of the non-centrality parameter is constrained by the size of the shape parameter. This rank condition has been recently proved with different methods for the special case of symmetric, positive semidefinite matrices (Letac and Massam (2011) and Graczyk, Malecki and Mayerhofer (2016)). 
\end{abstract}

\maketitle

%%%%%%%%%%%%%%%%%%%%%%%%%%%%%%%%%%%%%%%%%%%%%%%%%%%%%%%%%%%%%%%%%%%%%%%%

%    Templates for common elements of a journal article; for additional
%    information, see the AMS-LaTeX instructions manual, instr-l.pdf,
%    included in the TRAN author package, and the amsthm user's guide,
%    linked from http://www.ams.org/tex/amslatex.html .

%    Section headings

%
%%-----------------------------------------------------------------------
%% End of tran-l-template.tex
%%-----------------------------------------------------------------------
%
%\RequirePackage[OT1]{fontenc}
%\RequirePackage{amsthm,amsmath}
%\RequirePackage[numbers]{natbib}
%\RequirePackage[colorlinks,citecolor=blue,urlcolor=blue]{hyperref}
%
%
%
%\usepackage{amssymb}
%
%
%\newtheorem{theorem}{Theorem}[section]
%\newtheorem{corollary}[theorem]{Corollary}
%\newtheorem{assumption}[theorem]{Assumption}
%\newtheorem{definition}[theorem]{Definition}
%\newtheorem{example}[theorem]{Example}
%\newtheorem{lemma}[theorem]{Lemma}
%\newtheorem{proposition}[theorem]{Proposition}
%\newtheorem{remark}[theorem]{Remark}

%\numberwithin{equation}{section}

\renewcommand{\labelenumi}{\rm{(\roman{enumi})}}
\renewcommand{\theenumi}{\rm{(\roman{enumi})}}
\newcommand{\eps}{\varepsilon}
\newcommand{\tr}{\operatorname{tr}}
\newcommand{\rank}{\operatorname{rank}}
\newcommand{\diag}{\operatorname{diag}}

\section{Introduction}
For specific parameter instances, the Wishart distribution arises as the distribution of the covariance of samples from the multivariate normal distribution and was discovered as such by Wishart \cite{wishart1928generalised}. The Wishart distribution $\Gamma(\beta,\Sigma)$ on the cone of positive semidefinite $m\times m$ matrices 
$\mathcal S_m^+$, can be defined via its Laplace transform\footnote{The characteristic function
as computed for half-integers $\beta$ by \cite[p.444]{muirhead2009aspects}, and used in \cite{peddada1991proof} as defining equation is cumbersome, due to the use of the complex power function $z^\beta$, $\beta>0$, and is therefore avoided throughout this paper.}
\[
\Phi(u; \beta,\Sigma):=\int_{\mathcal S_d^+}e^{-\tr(u\xi)}p(d\xi)=\det(I_m+2u\Sigma)^{-\beta},\quad u \in \mathcal S_m^+,
\]
where $p\sim \Gamma(\beta,\Sigma)$. Here one assumes 
non-degenerate scale parameter $\Sigma\in\mathcal S_m^{++}$ (the positive definite matrices) and a non-negative shape parameter $\beta \geq 0$. From its definition, one can see that the Wishart distribution is a natural multivariate generalization of the Gamma distribution to the cone of positive semidefinite matrices. Unlike this
one-dimensional distribution, the Wishart distribution does not exist for all natural parameter choices. In fact, it is well known that $\Gamma(\beta,\Sigma)$ exists, if and only if 
\begin{equation}\label{eq: gindikin}
\beta\in \Lambda=\{0,\frac{1}{2},\dots,\frac{m-2}{2}\}\cup [\frac{m-1}{2},\infty).
\end{equation}
The set $\Lambda$ is often referred to as Gindikin set, in honor of Gindikin who studied Gamma distributions on homogeneous cones \cite{Gindikin}.

Graczyk, Malecki and Mayerhofer \cite{gmm} prove that in the more general case of the non-central Wishart distribution $\Gamma(\beta,\Sigma;\Omega)$ (with non-centrality parameter $\Omega\in\mathcal S_m^+$)
whose Laplace transform is given by
\[
\Phi(u;\beta,\Sigma,\Omega)=\det(I_m+2u\Sigma)^{-\beta} e^{-2\tr(\Omega u\Sigma( I_m+2 u\Sigma)^{-1})},\quad u\in S_m^+,
\]
in addition to \eqref{eq: gindikin}, the rank condition must hold:
\begin{equation}\label{eq: rank condition}
\text{If}\;\beta<\frac{m-1}{2},\text{ then }\; \rank(\Omega)\leq 2\beta.
\end{equation}
Their proof uses the modern theory of affine Markov semigroups with state space $\mathcal S_m^+$ \cite{CFMT}, and studies the action of these on elementary symmetric functions. Another, perhaps more elementary but rather technical proof of Letac and Massam \cite{letac2011existence} requires the precise decomposition of $\Gamma(\frac{m-1}{2},\Sigma;\Omega)$ into its singular and absolute continuous parts, and a detailed analysis of the support of non-central Wishart distributions.

This paper advances the very first approach to the existence issue of non-central Wishart distributions by Peddada and Richards \cite[Theorem 2]{peddada1991proof} who use properties of spherical polynomials with matrix argument to derive
the weaker condition \eqref{eq: gindikin}, albeit only for the $\rank(\Omega)=1$ case. Besides, they were not aware of the rank
condition \eqref{eq: rank condition}, which had first been conjectured by the author in a talk at a CIMPA workshop in Hammamet in 2011 (this rank condition is
termed ``Mayerhofer conjecture" in \cite{letac2011existence} and ``Non-central Gindikin set conjecture" in \cite{gmm}).

%The proof of Theorem 2 in \cite{peddada1991proof} faces two obstructions from obtaining insights into the arbitrary rank case. First, they only know non-negativity of
%the generalized binomial coefficients ${\kappa\choose\sigma}$ for $\sigma=(s)$, but require that these
%are strictly positive for arbitrary multi-index $\sigma$. Second, they overlook
%in \cite[Equations (12) and (13)]{peddada1991proof} that $(\beta)_\sigma=0$ for certain $\beta\in \{0,1,\dots,\frac{m-2}2)$ and $\sigma$ which actually implies that $\alpha=0$ is impossible, when $\rank(\Omega)=1$, which is nothing but the rank condition \eqref{eq: rank condition} for the rank one case. 

The refinement of Peddada and Richards' method produces the shortest and most elementary proof of the statements in \cite{gmm,letac2011existence}: 

\begin{theorem}\label{th: main}
Let $\beta\geq 0$, $\Omega\in\mathcal S_m^+$ and $\Sigma\in\mathcal S_m^{++}$. If $\Gamma (\beta,\Sigma;\Omega)$ exists, then conditions \eqref{eq: gindikin} and \eqref{eq: rank condition} must hold.
\end{theorem}
It turns out that Peddada and Richards' method is perfectly suited for an application in the general setting of symmetric cones. The appropriate
generalization of Theorem \ref{th: main} is developed in Section \ref{sec: symmetric cones} (Theorem \ref{th: main2}). This result
offers also a solution to \cite[Conjecture 1]{peddada1991proof} in the Hermitian case (but goes beyond it, including a rank condition), while it of course comprises all other irreducible symmetric cones: The Lorentz cones 
and the positive self-adjoint cones of arbitrary dimensions (over the reals, complex numbers, quaternions) as well as the exceptional $27$ dimensional cone of $3\times 3$ matrices
over the octonions.
\subsection*{Program}
In Section \ref{sec: prep} some facts on {\it generalized binomial coefficients} are elaborated, in particular their non-negativity, as well as
strict positivity in certain cases (Theorem \ref{thm1: conj1}). These coefficients arise in expansions of zonal polynomials which in turn are normalized Jack polynomials, with a parameter depending on the geometry of the particular symmetric cone (Theorem \ref{th: jack and zonal}). The above stated Theorem \ref{th: main} is proved in Section \ref{proof th1}.This result is then extended to the irreducible symmetric cone setting in Section \ref{sec: symmetric cones} (Theorem \ref{th: main2}). For the convenience of the reader, who might not be familiar with the algebraic structure of Euclidean Jordan Algebras, sections \ref{sec: matrix case} and \ref{sec: symmetric cones} are kept independent and self-contained. 

 In Appendix \ref{appendix: C} some facts of symmetric cones, and algebraic properties for Euclidean Jordan algebras are summarized. Essential formulas for the so-called contiguous binomial coefficients are given in Appendix \ref{appendix: A}. Finally, in Appendix \ref{appendix: B} we elaborate the maximal domain of the moment generating function of Wishart distributions. Furthermore, the action of linear automorphisms as well as of the natural exponential family on non-central Wishart distributions, which is required for the proof of Theorem \ref{th: main2}, is studied.

\section{Positivity of Generalized Binomial Coefficients}\label{sec: prep}

Let $m\in\mathbb N$, and $\sigma=(\sigma_1,\dots,\sigma_m)$ be a multi-index. The degree of a multi-index $\sigma\in\mathbb N_0^m$ is $\vert \sigma\vert:=\sigma_1+\dots+\sigma_m$.  All multi-indices that we consider here are in non-increasing order. The length $l(\sigma)$ is the last index $i$ for which $\sigma_i\neq 0$. The set of length $m$ multi-indices is denoted by $\mathcal E_m:=\{\sigma \in\mathbb N_0^m\mid \sigma_1\geq \sigma_2\geq\dots\geq \sigma_m\}$. The $i^{\text{th}}$ contiguous multi-index is defined by $\sigma^{i}:=(\sigma_1,\dots,\sigma_i+1,\dots,\sigma_m)$. Denote by $1^m$ the $m$-vector $(1,\dots,1)$. 

%A partial order $\nu\subseteq \lambda$ on multi-indices is defined by $\nu_i\leq \lambda_i$ for each $i$. 

For a given $\alpha>0$, the Jack polynomials $J_\kappa(\,;\alpha)$ indexed by $\kappa\in \cup_{m\geq 0}\mathcal E_m$ are the unique normalized symmetric polynomials in several variables with coefficients in the field $\mathbb Q(\alpha)$, which satisfy both a triangular, and an orthogonality property 
with respect to the a certain scalar product (for details, see \cite[Chapter VI, (4.5)]{macdonald1998symmetric}). This paper uses
another characterization put forward by R. Stanley (\cite[Theorem 3.1]{stanley1989some}).
\begin{theorem}
The polynomial $J_\kappa(\,;\alpha)$ is an eigenfunction of the Laplace-Beltrami-type operator 
\begin{equation}\label{eq: eigen alpha}
D(\alpha):=\frac{\alpha}{2}\sum_{i=1}^m t_i^2\frac{\partial^2}{\partial t_i^2}+\sum_{\substack{1\leq i,j\leq m\\ i\neq j}}\frac{t_i^2}{t_i-t_j}\frac{\partial}{\partial t_i},
\end{equation}
associated with the eigenvalue
\[
e_\kappa(\alpha):=\alpha \sum_{i=1}^m \kappa_i\frac{(\kappa_i-1)}{2}-\sum_{i=1}^m(i-1)\kappa_i+(m-1)\vert \kappa\vert.
\]
There are no further eigenfunctions linearly independent from the $J_\kappa$'s.
\end{theorem}
\begin{remark}
\cite{stanley1989some} gives a formula in terms of conjugate indices. The quoted form in \eqref{eq: eigen alpha} can be obtained by using the computations in the proof of \cite[Theorem 7.2.2]{muirhead2009aspects} together with the proof of \cite[Theorem 3.1]{stanley1989some}, where one needs only to prove that $D_\alpha t_1^{\kappa_1}\dots t_m^{\kappa_m}=e_\kappa(\alpha) t_1^{\kappa_1}\dots t_m^{\kappa_m} +$ lower order terms.
\end{remark}

The {\bf generalized binomial coefficients} $\pmb{{\kappa \choose \sigma}_\alpha}$ are implicitly defined\footnote{The generalized binomial coefficients are well-defined, since the Jack polynomials $J_\lambda(;\alpha)$  with $\lambda_1+\lambda_2+\dots=n$ form a basis of homogeneous symmetric functions of
degree $n$ with coefficients  in $\mathbb Q(\alpha)$ (\cite{stanley1989some}, p.83). Note that in \eqref{def: bin} the left side is a symmetric polynomial of the same degree as $J_\kappa(\;;\alpha)$ and the coefficients $J_\kappa(1^m;\alpha)$ are strictly positive, see \cite[Theorem 5.4]{stanley1989some}.} by the identity
\begin{equation}\label{def: bin}
\frac{J_\kappa(t+1^m;\alpha)}{J_\kappa(1^m;\alpha)}=\sum_{s=0}^k\sum_{\vert\sigma\vert=s}{\kappa \choose \sigma}_\alpha \frac{J_\sigma(t;\alpha)}{J_\sigma(1^{m};\alpha)}.
\end{equation}

The coefficients \eqref{def: bin} are rational numbers by construction. Peddada and Richards remark \cite{peddada1991proof},
\begin{quote}
``In the real case no explicit formula is available for the generalized binomial coefficients, and it does not even appear to be known whether they are nonnegative always."
\end{quote}

We will need more and less: It turns out that we need strict positivity, yet only for a special class of coefficients, the contiguous ones.
\begin{theorem}\label{thm1: conj1}
${\kappa \choose \sigma}_\alpha$ is a nonnegative rational number for any $\kappa,\sigma\in\mathbb N_0^m$. Furthermore, the contiguous coefficients ${\sigma^{i}\choose \sigma}_\alpha$ are strictly positive.
\end{theorem}

\begin{proof}
By \cite[Lemma 4]{kaneko1993selberg} (and for zonal polynomials, where $\alpha=2$, \cite[Exercise 7.13]{muirhead2009aspects}),
\begin{equation}\label{eq: 7.13}
\sum_{i=1}^m{\sigma^{i}\choose \sigma}_\alpha{\kappa \choose \sigma^{i}}_\alpha=(\vert \kappa\vert -\vert \sigma\vert) {\kappa \choose \sigma}_\alpha.
\end{equation}If $\vert\kappa\vert=\vert \sigma\vert$, there is nothing to prove, due to the fact that (cf.~\cite[p.~268, (5)]{muirhead2009aspects})
\begin{equation}\label{dirac}
{\kappa \choose \sigma}_\alpha=\delta_{\kappa\sigma}\in \{0,1\}.
\end{equation}
The contiguous coefficients
${\sigma^{i}\choose \sigma}_\alpha$ can be explicitly computed for each $i=1,\dots,m$. It follows immediately from the formula in
Lemma \ref{lem: almost clear pos} below that
\begin{equation}\label{eq: contiguous}
{\sigma^{i}\choose \sigma}_\alpha\geq 0.
\end{equation}

Induction over $d=\vert \kappa\vert-\vert \sigma\vert$ is used. For any $\kappa,\sigma$ and $d=0,1$, the statement holds, see \eqref{eq: contiguous}, \eqref{dirac}). These cases
form the induction basis.

Supposing the induction hypothesis
\begin{equation}
{\mu \choose \nu}_\alpha\geq 0,\;\text{for all }\mu,\nu \;\text{satisfying}\;\vert\mu\vert-\vert \sigma\vert\leq d-1
\end{equation}
it follows that
\begin{equation}\label{eq: reduced bin1}
{\kappa \choose \sigma^{i}}_\alpha\geq 0
\end{equation}
for each $i=1,\dots,n$, because $\vert \kappa\vert-\vert\sigma^i\vert=d-1$.

Hence by \eqref{eq: 7.13}, \eqref{eq: contiguous} and \eqref{eq: reduced bin1},
\[
{\kappa \choose \sigma}_\alpha=\frac{1}{k-s}\left(\sum_{i=1}^m {\sigma^{i}\choose \sigma}_\alpha {\kappa \choose \sigma^{i}}_\alpha\right)\geq 0,
\]
and thus the first part of the statement is proved. The strict positivity of the contiguous coefficients follows directly from their explicit expressions provided by
Lemma \ref{lem: almost clear pos}.
\end{proof}

\section{The Wishart Distribution on Positive Semidefinite Matrices}\label{sec: matrix case}
Denote by $\mathcal S_m$ the symmetric $m\times m$ matrices.

Recall the definition of zonal polynomials of matrix argument. For a strictly positive constant $C_\kappa(I_m)$,
\begin{equation}\label{eq: def zonal}
C_\kappa(x)=C_\kappa(I_m)\Phi_\kappa(x),\quad \Phi_\kappa(x)=\int_{u\in \mathbb O(m)} \Delta_\kappa(u x u^*)du, \quad x\in\mathcal S_m,
\end{equation}
where $du$ is the Haar measure on the orthogonal group $\mathbb O(m)$, and for $\kappa_1\in\mathcal E_m$
\[
\Delta_\kappa(x)=\Delta_1(x)^{\kappa_1-\kappa_2}\Delta_2(x)^{\kappa_2-\kappa_3}\dots \Delta_m(x)^{\kappa_m},
\]
where $\Delta_i$ is the $i^{\text{th}}$ subminor of the matrix $x\in\mathcal S_m$.
\begin{lemma}\label{lem: zonal}
$C_{\kappa}(\Omega)\geq0$. Furthermore, if $\rank(x)=k$ and $l(\kappa)=k$, then $C_\kappa(x)>0$.
\end{lemma}
\begin{proof}
$C_{\kappa}(\Omega)\geq0$, because of the nonnegative integrand in the very definition of the zonal polynomials, eq.~\eqref{eq: def zonal}. Concerning strict positivity, note that the non-negative map $u\mapsto \Delta_\kappa(u x u^\top)$ is strictly positive in a non-empty open neighborhood $U_0$ of the identity $u=I_m$. Since $\mathbb O(m)$ is a compact set, it can be covered by finitely many translates of $U_0$. The translation invariance of the Haar measure implies
that strictly positive mass is assigned to $U_0$, whence
\[
 \Phi_\kappa(x)\geq \int_{u\in \mathbb O(m)\cap U_0} \Delta_\kappa(u x u^*)du>0.
\]
\end{proof}

\begin{lemma}\label{lem: red ext}
Let $\beta\geq 0$, $\Omega\in\mathcal S_m^+$ and $\Sigma\in\mathcal S_m^{++}$. If the Wishart distribution $\Gamma(\beta,\Sigma;\Omega)$ exists, so does $\Gamma(\beta,\Sigma;t\Omega)$ exist for any $t\geq 0$. 
\end{lemma}
Note that for $t=0$ the standard Wishart distribution  is recovered.
\begin{proof}
\cite[Lemma 3.5 (iii)]{gmm} (replace $\Omega\leftrightarrow\Omega$, $\Sigma\leftrightarrow2\Sigma$, $\widetilde \Omega=t\Omega$,
$\widetilde \Sigma\leftrightarrow 2\Sigma$). 
\end{proof}

\subsection{Proof of Theorem \ref{th: main}}\label{proof th1}
\begin{proof}
Suppose $\Gamma(\beta,\Sigma;\Omega)$ exists. By Lemma \ref{lem: red ext} this is implies
the existence of $\Gamma(\beta,I_m;t\Omega)$, for each $t\geq 0$. Assume a random matrix $S(t)\sim \Gamma(\beta,I_m;t\Omega)$, each $t\geq 0$, as random variables on a probability space $(\Omega,\mathcal F,\mathbb P)$ with expectation operator $\mathbb E$.

Denote by $(a)_\kappa$ the ``generalized Pochhammer symbol"
\[
(a)_\kappa:=\prod_{i=1}^{l(\kappa)} (a-\frac{1}{2}(i-1))_{\kappa_i},
\]
where the classical Pochhammer symbol is used, which is defined for $b\in\mathbb R$ as
\begin{equation}\label{eq: ph}
 (b)_0:=1,\quad (b)_k:=b (b-1)\dots (b-k+1),\quad k\in\mathbb N.
\end{equation}

By \cite[Equations (12)]{peddada1991proof}, for any $\kappa\in\mathcal E_m$ the following moment formula holds (note the slightly rewritten formula)
\begin{equation}\label{eq: moment}
\mathbb E C_\kappa(S(t))=C_\kappa(I_m)\sum_{\sigma}{\kappa \choose \sigma}\frac{(\beta)_\kappa}{(\beta)_\sigma}\frac{C_\sigma(t\Omega)}{C_\sigma(I_m)}.
\end{equation}
If $\beta\geq (m-1)/2$, nothing need to be shown. Suppose therefore $\beta<(m-1)/2$, and $k=\rank(\Omega)$. If $k=0$, then $\Omega=0$, and there is nothing to show. To outline a rigorous proof for \cite[Theorem 2]{peddada1991proof} of Peddada and Richards, and their result's extension including the rank condition, 
the rank one case is shown first: So assume $\rank(\Omega)=1$ and $m\geq 2$. For $\kappa=(1,1)=1^2$, $C_{\kappa}(\Omega)=0$, and therefore eq.~\eqref{eq: moment} yields
\begin{align*}
\mathbb E C_{1^2}(S(t))&=C_{1^2}(I_m)\left((\beta)_{1^2}+\frac{(\beta)_{1^2}}{(\beta)_1}\frac{C_1(t\Omega)}{C_1(I_m)}\right)\\&=C_{1^2}(I_m) \left(\beta(\beta-1/2)+(\beta-1/2){\kappa \choose 1}\frac{C_1(t\Omega)}{C_1(I_m)}\right).
\end{align*}
Since each zonal polynomail value is strictly positive (Lemma \ref{lem: zonal}), $\beta\in [0,1/2)$ is impossible, as it leads to strictly negative expectation, whence $2\beta\geq 1=\rank(\Omega)$.

For the general case $\rank(\Omega)=k$, consider $\kappa=1^{k+1}$. Then
\begin{align*}
\mathbb E C_\kappa(S(t))&=C_\kappa(I_m)\left((\beta)_{1^{k+1}}+(\beta)_{1^k}/\beta{\kappa \choose 1}\frac{C_{1}(t\Omega)}{C_{1}(I_m)}+\dots\right.\\&\left.\qquad+(\beta-k/2){\kappa \choose 1^{k}}\frac{C_{1^k}(t\Omega)}{C_{1^k}(I_m)}\right).
\end{align*}
The last summand is the leading coefficient in $t$. Furthermore, the contiguous coefficient
${\kappa\choose 1^k}$ in this last summand is strictly positive, due to Lemma \ref{lem: almost clear pos}.
Hence $\beta\in [0,k/2)$ implies that $\mathbb E C_\kappa(S(t))<0$, a mere impossibility. Whence $2\beta\geq k$.

It remains to show that $2\beta \in \{k/2,\dots,(m-2)/2\}$. For small $t$, the zero order coefficient in \eqref{eq: moment} dominates all other coefficients, implying thus
\[
(\beta)_\kappa \geq 0
\]
for each $\kappa$. In particular, for each $1\leq l\leq m$, and $\kappa=1^{l+1}$, 
\[
(\beta)_\kappa=\beta (\beta-1/2)\dots (\beta-l/2)\geq 0,
\]
which is violated, unless $\beta\in  \{k/2,\dots,{m-2}/2\}$ (for instance, if $\beta\in (k/2,(k+1)/2)$,
take $l=k+1$, and thus $(\beta)_\kappa<0$).
\end{proof}

\section{The Wishart Distribution on Symmetric Cones}\label{sec: symmetric cones}

% Perhaps useful: A. S. Yasamin, [...] Wishart Models... in Algebraic Methods in Statistics and Probability II

Let $C$ be an irreducible symmetric cone in a finite dimensional vector space $V$, $\dim V=n$.  We denote by $r$ the rank of the associated simple Euclidean Jordan algebra, with unit element $e$, and let $d$ its Peirce invariant. The inner product on $V$ is chosen as $\langle x,y\rangle:=\tr(xy)$. For a brief review on symmetric cones see Appendix \ref{appendix: C}.

The generalized Pochhammer symbol, parameterized in $\alpha>0$ (whose dependence is suppressed in the following) is given by
\[
(a)_\kappa:=\prod_{i=1}^{l(\kappa)} (a-\frac{1}{\alpha}(i-1))_{\kappa_i},
\]
where the classical Pochhammer symbol is used (see eq.~\eqref{eq: ph}).

\subsection{Zonal Polynomials: Definition and basic properties}
Let $K=O(V)\cap G$, where O(V) is the orthogonal group of $V$ and $G$ is the connected component of the identity in the linear automorphism group
of $C$. Let $dk$ be the normalized Haar measure on the compact abelian group $K$. The spherical polynomials
are defined by
\begin{equation}\label{def: zo1}
\Phi_\kappa(x):=\int_K \Delta_\kappa(kx) dk,\quad \kappa\in\mathcal E_m,\quad m\geq 1,
\end{equation}
where $\Delta_\kappa$ denotes the generalized power (see Appendix \ref{zonal}). Up to a postive constant\footnote{The constant $d_\kappa=1/\|\Delta_\kappa\|_\Sigma$ is computed explicitly in \cite[Proposition XI.4.1 and Proposition XI.4.3]{Faraut}.}
\begin{equation}\label{zonal coefficient}
c_\kappa=d_\kappa\frac{\vert \kappa\vert!}{\left(\frac{n}{r}\right)_\kappa}
\end{equation}
they are equal to the zonal polynomials:
\begin{equation}\label{def: zo2}
Z_\kappa(x)=c_\kappa \Phi_\kappa(x).
\end{equation}

This normalization is tailored in such a way that
\begin{equation}\label{eq: zonal identity}
\tr(x)^k=\sum_{\vert \kappa\vert=k} Z_\kappa(x)
\end{equation}
and thus \cite[Proposition XII.1.3 (i)]{Faraut}
\[
e^{tr(x)}=\sum_{\kappa}d_\kappa\frac{d_\kappa}{\left(\frac{n}{r}\right)_\kappa}\Phi_\kappa(x),  \quad x\in V.
\]

An immediate consequence of a functional relationship for zonal polynomials \cite[Corollary XI.3.2]{Faraut} is the following 
\begin{lemma}\label{lem 1x}
\begin{equation}
\int_K Z_\kappa(P(y^ {1/2})kx)dk=\frac{Z_\kappa(y)Z_\kappa(x)}{Z_\kappa(e)}.
\end{equation}
\end{lemma}
%We further remind that \cite[Proposition XII.1.3 (ii)]{Faraut} implies for any $\beta>0$
%\begin{lemma}
%\[
%\Delta (A+e)^{-\beta}=\sum_{k=0}^\infty\sum_{\kappa:\vert \kappa\vert=k} (-1)^{k}\frac{ (\alpha)_\kappa}{k!}Z_\kappa(A).
%\]
%\end{lemma}

\subsection{Zonal Polynomials are normalized Jack symmetric functions}

The functions $Z_\kappa$ are special cases of normalized Jack polynomials, when considered as functions of the eigenvalues of $x$:
\begin{theorem}\label{th: jack and zonal}
Let $\alpha=2/d$, where $d$ is the Peirce invariant. Then for any $\kappa\in\mathcal E$, 
\begin{equation}\label{eq: rel}
\Phi_\kappa(x)=  \frac{J_\kappa(\lambda(x);2/d)}{J_\kappa(1_{l(\kappa)};2/d)},
\end{equation}
where $\lambda(x)=(\lambda_1,\dots,\lambda_r)$ are the eigenvalues of $x$.
\end{theorem}
\begin{example}
For $d=1$, $k=2$ and $m=2$, $Z_\kappa=C_\kappa$, and
\begin{equation}\label{eq: zonal sum}
\sum_{\vert \kappa\vert=k}C_\kappa(z)=\tr(z)^k=(\lambda_1+\lambda_2)^2=(\lambda_1^2+\lambda_2^2+\frac{2}{3}\lambda_1\lambda_2)+\frac{4}{3}\lambda_1\lambda_2,
\end{equation}
where ($\lambda_1,\lambda_2$) are the eigenvalues of the symmetric positive semidefinite matrix $z$. On the other hand, the Jack polynomials of
degree $2$ are given by (\cite[Table 1, $\alpha=2$]{MOPS})
\begin{equation}\label{eq: jack formulas}
J_{(1,1)}(\lambda;\alpha=2)=2 \lambda_1\lambda_2,\quad J_{(2)}(\lambda;\alpha=2)=3 (\lambda_1^2+\lambda_2^2)+2 \lambda_1\lambda_2.
\end{equation}
By equation \eqref{eq: rel} the zonal polynomials are multiplies of the Jack polynomials, and by comparing \eqref{eq: jack formulas} and \eqref{eq: zonal sum}
we must have the following relationship
\[
C_{(1,1)}=\frac{2}{3}
J_{(1,1)}(,;\alpha=2),\quad C_{(2)}=\frac{1}{3}
J_{(2)}(,;\alpha=2).
\]
\end{example}
A consequence of Theorem \ref{th: jack and zonal} is the following crucial link between the generalized binomial coefficients
in the context of symmetric cones, and those defined via Jack polynomials:
\begin{remark}\label{rembrandt}
By the definition \eqref{def: zo1}, we have $\Phi_\kappa(e)=1$. Therefore $c_\kappa=Z_\kappa(e)$. The generalized binomial coefficients of \cite[p. 343]{Faraut}, are defined by the relationship
\[
\Phi_\mu(e+x)=\sum_{\vert \nu\vert\leq \vert \mu\vert}{ \mu \choose \nu}\Phi_\nu(x).
\]
Since $\Phi_\kappa(e)=1$, for each $\kappa$, we conclude that 
\[
{ \mu \choose \nu}={\mu \choose\nu}_{2/d},
\]
where the latter are the generalized binomial coefficients \eqref{def: bin} for Jack polynomials of parameter $\alpha=2/d$.
\end{remark}

\begin{proof}[Proof of Theorem \ref{th: jack and zonal}]
By \cite[Proposition VI.4.2]{Faraut}, the Laplace-Beltrami operator is of the form
\begin{align*}
L\Phi_\kappa(x)&=\sum_{i=1}^r \lambda_i^2 \frac{\partial^2 \Phi_\kappa}{\partial \lambda_i^2}+d\sum_{i\neq j}\frac{\lambda_i \lambda_j}{\lambda_i-\lambda_j}\frac{\partial \Phi_\kappa}{\partial \lambda_i}+\frac{n}{r}\sum_{i=1}^r \lambda_i \frac{\partial\Phi_\kappa}{\partial \lambda_i}\\
&=\sum_{i=1}^r \lambda_i^2 \frac{\partial^2 \Phi_\kappa}{\partial \lambda_i^2}+d\sum_{i\neq j}\frac{\lambda_i^2}{\lambda_i-\lambda_j}\frac{\partial \Phi_\kappa}{\partial \lambda_i}+\left(\frac{n}{r}-(r-1)d\right)\sum_{i=1}^r \lambda_i \frac{\partial\Phi_\kappa}{\partial \lambda_i},
\end{align*}
where $(\lambda_1,\dots,\lambda_r)$ are the eigenvalues of $x$. By the Euler equation for the homogeneous function $\Phi_\kappa$, we thus obtain
\begin{equation}\label{eq: eigen id}
L\Phi_\kappa-\left(\frac{n}{r}-(r-1)d\right)\vert\kappa\vert\Phi_\kappa= \frac{2}{\alpha }D(\alpha)\Phi_\kappa
\end{equation}
for $\alpha=2/d$.

Furthermore, by \cite[Exercise VII.2]{Faraut}, $\Delta_\kappa$ is an eigenfunction of the Laplace-Beltrami operator, and therefore also $\Phi_\kappa$ is an eigenfunction of $L$, and by \eqref{eq: eigen id},
it is an eigenfunction of $dD(2/d)$ also. Since (the symmetrization of) the monomial $t_1^{\kappa_1}\dots t_m^{\kappa_m}$ is the highest order term of
$\Phi_\kappa$, and since (see the proof of \cite[Theorem 7.2.2]{muirhead2009aspects} )
\[
dD(2/d)t_1^{\kappa}\dots t_m^{\kappa_m}=e_{\kappa}(\alpha) t_1^{\kappa_1}t_1^{\kappa}\dots t_m^{\kappa_m}+\text{  lower order terms},
\]
its eigenvalue equals $e_\kappa(2/d)$. Thus, by \cite[Theorem 3.1]{stanley1989some},
\[
\Phi_\kappa=\sum _{\mu\leq \kappa}g_{\mu \kappa} J_\mu(; 2/d).
\]
By the second part of the proof of \cite[Proposition 5.1]{stanley1989some}, it follows that
\[
\Phi_\kappa(\cdot)=f_{\kappa}(d) J_\kappa(\cdot;2/d),
\]
for some constant $f_{\kappa}(d)$.

By the definition \eqref{def: zo1}, we have $\Phi_\kappa(e)=1$, hence
\[
f_\kappa(d)=J_\kappa(1_{l(\kappa)};2/d)^{-1}.
\]
 \end{proof}

Using Remark \ref{rembrandt} we have the following generalization of \cite[Theorem 7.6.3]{muirhead2009aspects}:
\begin{proposition}\label{prop43}
\begin{equation}
\Delta (A+e)^{-\beta}\int_K e^{-\langle(A^{-1}+e)^{-1},k(\Omega)\rangle}dk=\sum_{k=0}^\infty \sum_{\kappa:\vert\kappa\vert=k}(-1)^k\frac{L_\kappa^\beta(-\Omega)Z_\kappa(A)}{k!},
\end{equation}
where the generalized Laguerre polynomials (indexed by $\kappa$) are defined by (\cite[p.343]{Faraut})
\[
L_\kappa^\beta(x):=\sum_{\sigma}{\kappa \choose \sigma}_{2/d}\frac{(\beta)_\kappa}{(\beta)_\sigma}\frac{Z_\sigma(-x)}{Z_\sigma(e)}.
\]
\end{proposition}
\begin{proof}
Use \cite[Exercise XV.3]{Faraut} with $\Omega=-u$, $x=-A$ and the homogeneity of $Z_\kappa$ of degree $\vert\kappa\vert=k$.
\end{proof}
%\begin{remark}
%\cite[Exercise VII.2]{Faraut} states that
%\[
%L\Delta_\kappa=(\kappa\mid \kappa-2\rho)\Delta_\kappa,
%\]
%where
%\[
%(\kappa\mid \kappa-2\rho)=\sum_{i=1}^r \kappa_i(\kappa_i-2\rho_i),\quad \rho_i=\frac{d}{4}(2i-r-1).
%\]
%This numerical value is inconsistent with the (proof of the) previous theorem, because then
%$\Phi_\kappa$ is an eigenfunction of $D(\alpha)$ with eigenvalue
%\[
%(\kappa\mid \kappa-2\rho)-\left(\frac{n}{r}-(r-1)d\right)\vert\kappa\vert \neq d e_\kappa(2/d).
%\]
%\end{remark}

\subsection{The Wishart distribution}

The Wishart distribution on symmetric cones (\cite[Corollary 3]{MassamNeher} and \cite[(3.17)]{casalisletac1}) is defined
%\{The Laplace transform in Neher is a product of two determinants $(\det \xi)^p (\det\xi-\theta)^{-p}$, where
%$\xi$ is a parameter (equals $\sigma^{-1}$ in our notation), and $\theta$ is the Laplace variable. In view of the lack of multiplicativity of the determinant, the product involves the quadratic %variation $P$, being equal $\det(e-P(\sqrt{\xi^{-1}}\theta))^{-p}$.} 
as follows:
\begin{definition}
Let $\beta\geq 0$ and $\Sigma$ be a regular element in $C$. The Wishart distribution $\Gamma (\beta,\Sigma)$, if exists, is defined via its Laplace transform $\Delta\left(e+P(\sqrt{\Sigma} u)\right)^{-\beta}$.
\end{definition}

An immediate consequence of the positivity result on Riesz distributions on symmetric cones
%\footnote{Gindikin \cite{Gindikin} shows this
%result first, albeit for a standardized form, but for the more general setting of homogeneous cones, which may lack self-duality. Unfortunately, his proof is cryptic and therefore difficult to %follow.} (\cite{casalis1994characterization}, \cite[Theorem VII.3.1]{Faraut}), 
which have Laplace transforms of the form
 $\varphi_\beta(u)=\Delta(u)^{-\beta}$, is the following:

\begin{proposition}
The Wishart distribution $\Gamma(\beta,\Sigma)$ is a positive measure, if and only if $\beta$ belongs to the (so-called) Wallach set
\begin{equation}\label{eq: wallach set}
\Lambda:=\{0,\frac{d}{2}, \dots, \frac{d}{2}(r-2)\}\cup \left[\frac{d}{2}(r-1),\infty\right).
\end{equation}
\end{proposition}
We shall extend this result for the more general case of non-central Wishart distributions, including a rank condition on the shape parameter. In accordance
with \cite[Proposition 5.5]{CKMT}, where this distribution arises naturally as time marginal of affine diffusion processes, we define the following
\begin{definition}
Let $\beta\geq 0$ and $\Sigma\in C^\circ$. The non-central Wishart distribution $\Gamma(\beta,\Sigma; \Omega)$, if exists, is defined via its Laplace transform 
\begin{equation}\label{FLT Mayerhofer Wishart}
\int e^{-\tr(u \xi)}\Gamma(\beta,\Sigma; \Omega)(d\xi):=\Delta\left(e+2P(\sqrt{\Sigma}) u\right)^{-\beta}e^{-\langle (u^{-1}+2\Sigma)^{-1},\Omega\rangle}.
\end{equation}
\end{definition}

\begin{lemma}\label{lem: zonal sym}
$Z_{\kappa}(\Omega)\geq0$. Furthermore, if $\rank(x)=k$ and $l(\kappa)=k$, then $Z_\kappa(x)>0$.
\end{lemma}
\begin{proof}
$Z_{\kappa}(\Omega)\geq0$, because of the nonnegative integrand in the very definition of the zonal polynomials, eq.~\eqref{def: zo1}--\eqref{def: zo2}. 

For the proof of strict positivity, note that the non-negative map $k\mapsto \Delta_\kappa(k x)$ is strictly positive in a non-empty open neighborhood $U_0$ of the identity map. Since $K$ is a compact set, it can be covered by finitely many translates of $U_0$. The translation invariance of the Haar measure implies
that strictly positive mass is assigned to $U_0$, whence
\[
 \Phi_\kappa(x)\geq \int_{u\in K\cap U_0} \Delta_\kappa(kx)dk>0.
\]
\end{proof}

\begin{theorem}\label{th: main2}
If $\Gamma(\beta,\Sigma; \Omega)$ exists, then the following conditions must hold:
\begin{enumerate}
\item $\beta\in \Lambda $, and 
\item if $\beta\in\{0,\frac{d}{2},\dots \frac{d}{2}(r-2)\}$, then  $d\rank(\Omega)\leq 2\beta$.
\end{enumerate}
\end{theorem}
\begin{proof}
If $2\beta\geq d(r-1)$, nothing need to be shown. Assume therefore $2\beta<d(r-1)$.

Suppose $S\sim \Gamma (\beta,\Sigma; \Omega)$. Then by the second part of Lemma \ref{mayx} \ref{x1}, $P(\sqrt{\Sigma^{-1}}S\sim\Gamma(\beta,e;\Omega')$,
where $\Omega'=P(\sqrt{\Sigma^{-1}})(\Omega)$ is of the exact same rank as $\Omega$, because $\Sigma$ is invertible. Second, by the first part of
Lemma \ref{mayx} \ref{x1}, using the map $P(\sqrt{se})$, we obtain the existence of $\Gamma(\beta,te;s\Omega')$, for each $s>0$.
By Lemma \ref{mayx} \ref{x2}, $\Gamma(\beta,e;\frac{s\Omega'}{s^2})$ exists, for any $s>0$. Setting $s=t$,
we see that $\Gamma(\beta,e;t \Omega')$ exists for any $t>0$.

As the Laplace transforms converge pointwise for $t\rightarrow 0$ to $\det(e+2u)^{-\beta}$, by L\'evy's continuity theorem for
Laplace transforms (\cite[Lemma 4.5]{CFMT}) the limit is
the Laplace transform of a probability measure (namely of $\Gamma(\beta,e)$). We conclude that $\Gamma(\beta,e;t\Omega')$
exists for any $t\geq 0$. Let $(\Omega,\mathcal F,\mathbb P)$ be a probability space that supports $S(t)\sim \Gamma(\beta,e; t\Omega')$,
for any $t\geq 0$, and denote by $\mathbb E$ the corresponding expectation operator.

Next, we establish a crucial moment formula: For any $\kappa\in\mathcal E$
\begin{equation}\label{eq: momentsc}
\mathbb E Z_\kappa(S(t))=Z_\kappa(e)L_\kappa^{\beta}(-t\Omega).
\end{equation}
To this end, we manipulate the following sum, using Lemma \ref{lem 1x} (second identity), equation \eqref{eq: zonal identity} (third identity), Lemma \ref{lem: properties} \ref{prop 3a} (fourth identity),
\begin{align*}
&\sum_{k=0}^\infty \sum_{\kappa: \vert\kappa\vert=k}(-1)^k\frac{Z_\kappa(A)\mathbb E[Z_\kappa(S(t))]}{k! Z_\kappa(e)}=\mathbb E\left[\sum_{k=0}^\infty \sum_{\kappa: \vert\kappa\vert=k}(-1)^k\frac{Z_\kappa(A)Z_\kappa(S(t))}{k! Z_\kappa(e)}\right]\\
&=\mathbb E\left[\sum_{k=0}^\infty \sum_{\kappa: \vert\kappa\vert=k}\frac{(-1)^k}{k!}\int_K Z_\kappa(P(A^{1/2}) kS(t)) dk\right]=\mathbb E\left[\int_K e^{-\tr(P(A^{1/2}) kS(t))}dk\right]\\
&=\mathbb E\left[\int_K e^{-\tr(AkS(t))}dk\right]=\int_K\mathbb E\left[ e^{-\tr(k^{-1}A S(t))}dk\right]\\
&=\int_K \Delta\left(e+k^{-1}A\right)^{-\beta}e^{-\langle ((k^{-1}A)^{-1}+ e)^{-1},t\Omega\rangle}dk\\&=\Delta\left(e+ A\right)^{-\beta}\int_K e^{-\langle ((k^{-1}A)^{-1}+ e)^{-1},t\Omega\rangle}dk\\
&=\Delta\left(e+ A\right)^{-\beta}\int_K e^{-\langle (A^{-1}+ e)^{-1},k(t\Omega)\rangle}dk\\
&=\sum_{k=0}^\infty \sum_{\kappa:\vert\kappa\vert=k}(-1)^k\frac{L_\kappa^\beta(-t\Omega)Z_\kappa(A)}{k! },
\end{align*}
the last identity being Proposition \ref{prop43}. By comparison of coefficients, we conclude that \eqref{eq: momentsc} holds.

Let $r>2$ and $\rank(\Omega)=k\geq 0$, then for $\kappa=1^{k+1}$, 
\begin{align*}
\mathbb E Z_\kappa(S(t))&=Z_\kappa(e)\left((\beta)_{1^{k+1}}+(\beta)_{1^k}/\beta{\kappa \choose 1}\frac{Z_{1}(t\Omega)}{Z_{1}(e)}+\dots\right.\\&\qquad\qquad\qquad\qquad\qquad\left.+(\beta- d k/2){\kappa \choose 1^{k}}\frac{Z_{1^k}(t\Omega)}{Z_{1^k}(e)}\right).
\end{align*}
The last summand is the leading coefficient in $t$. Furthermore, the contiguous coefficient
${\kappa\choose 1^k}$ in this last summand is strictly positive, due to Lemma \ref{lem: almost clear pos}.
Hence $\beta\in [0, dk/2)$ implies that $\mathbb E Z_\kappa(S(t))<0$, a mere impossibility. Whence $dk\leq 2\beta<d(r-1)$.

It remains to show that $2\beta \in \{dk,\dots,d(r-2)\}$. For small $t$, the zero order coefficient in \eqref{eq: momentsc} dominates all other coefficients, implying thus
\[
(\beta)_\kappa \geq 0
\]
for each $\kappa$. In particular, for each $k+1\leq l\leq r-1$, and $\kappa=1^{l+1}$, 
\[
(\beta)_\kappa=\beta (\beta-d\frac{1}{2})\dots (\beta-d\frac{l}{2})\geq 0,
\]
which is violated, if $\beta\in (d(l-1)/2,d l/2)$.

\end{proof}

\appendix

\section{Basic properties of symmetric cones}\label{appendix: C}
Let $V$ be an $n$ dimensional vector space over $\mathbb R$, equipped with an inner product $\langle\,,\,\rangle$ and let $C\subset V$ be a closed convex cone, with $C^\circ$ its interior. Let $GL(V)$
be the general linear group on $V$. $C$ is called homogeneous, if the linear automorphism group 
\[
G(C^\circ):=\{g\in GL(V)\mid g(C^\circ)=C^\circ\}
\]
 acts transitively on $C$, i.e., for each $c_1,c_2\in C$, there exists $g\in G(C^\circ)$ such that $g(c_1)=c_2$. $C$ is symmetric, if it is homogeneous and self-dual, i.e,
\[
C=\{x\in V\mid \langle x,u\rangle\geq 0\text { for any }u\in C\}.
\]
Since irreducible symmetric cones can be associated with simple Euclidean Jordan algebras in a unique way (\cite[Proposition III.4.5]{Faraut}),
we identify the vector space $V$ with its Euclidean Jordan algebra, and $C$ is the set of squares, $C=\{x^2\mid x\in V\}$. 
Let therefore $C$ be a symmetric cone, and $e$ be an identity element in $V$.

For $x\in V$, the map $L(x): V\rightarrow V$ is the left multiplication $y\mapsto xy$. For $x\in V$, the quadratic representation $P(x) : V\rightarrow V$
is then defined by
\begin{equation}\label{eq: def P}
P(x) y=2 L(x) (L(x)y)-L(x^2)y.
\end{equation}
In the case of associativity, we have due to
commutativity of multiplication,
\[
P(x)y=x^2y=xyx.
\]
Therefore, the definition of $P$ in \eqref{eq: def P} reflects the fact that multiplication is not associative for all Euclidean Jordan algebras.

Clearly the set of squares is homogeneous, because for any $x\in C$, we can write $x=y^2$ for some $y\in V$, and $P(y)e=y^2=x$.

Let $\mathbb R[\lambda]$ be the ring of polynomials over $\mathbb R$, and for fixed $x\in V$ consider the ring $\mathbb R[x]$ of polynomials over $\mathbb R$, evaluated at $x$.

The minimal polynomial of $x$ is the polynomial with leading coefficient $1$ that generates the ideal

\[
\mathcal J(x)=\{p\in\mathbb R[\lambda]\mid p(x)=0\},
\]
and then
\[
\mathbb R[x]=\mathbb R[\lambda]/\mathcal J(x).
\]

The degree of the minimal polynomial of $x\neq 0$ is precisely
\[
m(x)=\min\{k>0\mid \{e,x,\dots, x^k\} \text{ are linearly independent}\}.
\]
The rank of the Jordan algebra is defined as $r:=\max_{x\in V} m(x)$
and an element $x$ is called regular, if $m(x)=r$. The set of regular elements in $C$
is precisely the interior  $C^\circ$ of the cone.

By \cite[Proposition II.2.1]{Faraut}, there exist unique polynomials $a_1,\dots, a_r$, such that the minimal polynomial of every regular
element is given by
\[
f(\lambda;x)=\lambda^r-a_1(x)\lambda^{r-1}+\dots+(-1)^ra_r(x),
\]
and we denote $\det(x)=a_r(x)$ and $\tr(x)=a_1(x)$, the determinant and trace of $x$.

\subsection{Determinant, Trace and Quadratic Representation}
We use the trace as inner product by defining (any other symmetric bilinear form in a simple Euclidean Jordan algebra $V$ is a scalar multiple of the trace, see \cite[Theorem III.4.1]{Faraut})
\[
\langle x,y\rangle:=\tr(xy).
\]
This inner product is associative (\cite[Proposition II.4.3]{Faraut}), that is,
\begin{equation}\label{eq: ass}
\langle xy,z\rangle=\langle y, xz\rangle.
\end{equation}
In other words: the left multiplication is a self-adjoint operator. The following are frequently used in this paper:
\begin{lemma}\label{lem: properties}
\begin{enumerate}
\item \label{prop 1a} $(P(x)y)^{-1}=P(x^{-1}) y^{-1}$,
\item \label{prop 3a} $\langle P(x) y,z\rangle=\langle y, P(x)z\rangle$,
\item \label{prop 2a} $\det(P(x)y)=\det(x^2)\det(y)$.
\end{enumerate}
\end{lemma}
\begin{proof}
Property \ref{prop 1a} is \cite[Proposition II.3.3 (ii)]{Faraut}, and \ref{prop 2a} is \cite[Proposition III.4.2(i)]{Faraut}. Property \ref{prop 3a} follows from the fact that $L(x)$ is a self-adjoint operator, \eqref{eq: ass}.
\end{proof}

\subsection{Coordinates}
The rank $r$ and the Peirce invariant $d$ of a symmetric cone are linked to two important coordinate systems on symmetric cones, which are introduced below for a
better understanding of the paper (which is, however largely coordinate free). 

\subsubsection{Spectral Decomposition and rank}
$f\in V$ is called idempotent, if $f^2=f$, hence $f\in C$. Two idempotents $f,g$ are orthogonal, if $fg=0$. An idempotent is primitive, if it is not the sum
of two non-trivial idempotents. A set of idempotents is complete, if it sums to the unit element $e$. Any element $x\in V$ admits a spectral decomposition: There exists
a complete set of mutually orthogonal primitive idempotents $\{e_1,\dots, e_r\}$ (a so-called Jordan frame) and real numbers $\lambda_1,\dots,\lambda_r$ such that $x=\lambda_1 e_1+\dots+\lambda_r e_r$. The coefficients $\lambda_1,\dots,\lambda_r$ are called eigenvalues.

\subsubsection{Pierce Decomposition II and invariant}
Idempotents can only have eigenvalues $0,1/2$ or $1$. For $\lambda\in\{0,1/2,1\}$ and idempotents $c$, we define the linear subsaces $V(c,\lambda):=\{x\in V\mid cx=\lambda x\}$. 

Let $c_1,\dots,c_r$ be a Jordan frame. The space $V$ can be decomposed (\cite[Theorem IV.2.1]{Faraut})
in the following orthogonal sum
\[
V=\bigoplus_{i\leq j} V_{ij},
\]
where $V_{ii}=V(c_i,1)$ for $i=1,\dots, r$, and for $1\leq i<j\leq r$, $V_{ij}=V(p_i,1/2)\cap V(c_j,1/2)$. Furthermore, the orthogonal projections $P_{ij}$ on $V_{ij}$ satisfy
\begin{equation*}
P_{ii}=P(c_i), \quad P_{ij}=4L(c_i)L(c_j).
\end{equation*}

The dimension $d$ of $V_{ij}$, the Peirce invariant, 
is independent of $i,j$ and the Jordan frame. It satisfies (\cite[Corollary IV.2.6]{Faraut})
\begin{equation}
n=r+\frac{d}{2} r(r-1).
\end{equation}
For example, for $d=1$ (where $C$ is the space of positive semidefinite symmetric $r\times r$ matrices), we indeed have 
\[
n=\frac{(r+1)}{2}=r+\frac{1}{2} r(r-1).
\]
\subsection{Generalized Power Functions and Orthogonal Group}\label{zonal}

For $\kappa\in\mathcal E_m$, and $m\leq r$, the generalized power functions are defined as
\[
\Delta_\kappa(x):=\Delta_1^{\kappa_1-\kappa_2}(x)\Delta_2^{\kappa_2}(x)\dots\Delta_m(x)^{\kappa_m},
\]
where $\Delta_j$ for $j=1,\dots, r$ denotes the principal minors corresponding to the Jordan algebras
$V^{(j)}=V(c_1+\dots+c_j,1)=\{x\in V\mid x=x(c_1+\dots c_j)\}$, where $(c_1,\dots, c_r)$ is a Jordan frame. Clearly it holds that
$\Delta_\kappa(x)=\lambda_1^{\kappa_1}\dots \lambda_m^{\kappa_m}$, if $x=\lambda_1 c_1+\dots \lambda_r c_r$. 

The spherical polynomials $\Phi_\kappa$ (and thus the zonal polynomials $Z_\kappa$) arise as symmetrizations of generalized powers, see \eqref{def: zo1}. Let $O(V)$ be the orthogonal group defined by 
\[
O(V):=\{g\in GL(V)\mid \langle gc,gc\rangle=\langle c,c\rangle\text {  for all  } c\in V\}.
\]

Let $G\subset G(C^\circ)$ be the connected component of the identity, and set $K:=G\cap O(V)$. By \cite[p.5]{Faraut} it already acts transitively on $C$. Furthermore, according to \cite[Proposition I.1.9 and Proposition I.4.3]{Faraut}, there exists an element $e\in C$ such that
\begin{equation}
G(C^{\circ})\cap O(V)=G(C^\circ)_{e}:=\{g\in G(C^\circ)\mid ge=e\},
\end{equation}
the latter denoting the stabilizer of $e$,
and
\begin{equation}\label{eq stabilizer identity}
G_e=K:=G\cap O(V).
\end{equation}
A nice result concerns the automorphism group $\text{Aut}(V)$ (not to be confused with $GL(V)$), which is defined as 
\[
\text{Aut}(V):=\{A\in GL(V)\mid A(xy)=A(x) A(y)\}.
\]
By \cite[Theorem III.5.1]{Faraut}, we have
\[
\text{Aut}(V)_0=K,
\]
where $\text{Aut}(V)_0$ denotes the identity component of $\text{Aut}(V)$.

\subsection{Zonal polynomials and Symmetry}
Since $K$ is a compact abelian group, there exists a Haar measure $dk$ on it. We use its unique normalization. 

The zonal polynomials defined via \eqref{def: zo1} are symmetric functions, in the sense that they are $K$ invariant, by constructions. This implies the following:
\begin{lemma}
$Z_\kappa(x)$ depend on the eigenvalues $\lambda_1,\dots,\lambda_r$ of $x$ only.
\end{lemma}
\begin{proof}
Let $x,y \in V$, with Jordan frames $c_1,\dots, c_r$ and $d_1,\dots, d_r$ and eigenvalues $\lambda_1,\dots,\lambda_r$
such that
\[
x=\sum_i \lambda_i c_i,\quad y=\sum_j \lambda_j d_j.
\]
By Theorem \cite[Theorem IV.2.5]{Faraut} there exists an automorphism $A$ such that $Ac_i=d_i$ for $i=1,\dots, r$. Since $A$ is an algebra automorphism,
and $(c_i)_{i=1}^r$ is a complete and orthogonal system of idempotents, we have
\[
A(e) A(e)=A(e^2)=A((c_1+\dots c_r)^2)=A(c_1^2+\dots +c_r^2)=A(c_1+\dots + c_r)=A(e)
\]
and therefore $A(e)=e$. Since $\text{Aut}\subseteq GL(V)$ by definition and $A(e)=e$, we conclude by \eqref{eq stabilizer identity} that $A\in K$. Hence by $K$-invariance, $Z_\kappa(x)=Z_\kappa(Ax)=Z_\kappa(y)$.
\end{proof}

\section{Ferrers Diagrams}\label{appendix: A}
The material of this section is a slight adaption of notation and results in
 \cite{stanley1989some, kaneko1993selberg}.

Any multi-index $\kappa=(\kappa_1,\dots,\kappa_m)$ with $\kappa_1\geq\dots\geq \kappa_m$ can be identified with its Ferrers diagram which is the specific partition
\[
\mathcal P_\kappa:=\{s=(i,j)|1\leq i\leq l(\kappa),1\leq j\leq \kappa_i\}.
\]
In this diagram, $\kappa_i$ are the number of squares in each row $i$. The number $\kappa^j$ of squares in each column $j$, where $1\leq j\leq \kappa_1$, 
are 
\[
\kappa^j=\#\{i\mid (i,j)\in\mathcal P_\kappa\}.
\] 
For $s=(i,j)\in \mathcal P_\kappa$, write $i(s):=i$ and $j(s):=j$ to indicate its column or row label.

The arm-length, for each $s$ in the diagram, is defined as the number of squares to the right of $s$, that is
\[
a_\kappa(s):=\kappa_i-j(s),
\]
and the leg-length is the number of squares below $s$,
\[
l_\kappa(s):=\kappa^{j(s)}-i(s).
\]

\begin{definition}
The upper hook- length is defined by
\[
h_\kappa^*(s):=l_\kappa(s)+\alpha(1+a_\kappa(s))
\]
and the lower hook-length is
\[
h^\kappa_*(s):=l_\kappa(s)+1+\alpha a_\kappa(s).
\]
\end{definition}

The following holds by \cite[Proposition 2]{kaneko1993selberg},
\begin{lemma}\label{lem: almost clear pos}
\[
{\sigma^i \choose \sigma}_\alpha=j_\sigma(\alpha)^{-1}\prod_{s\in \mathcal P_\sigma}A_{\sigma^i} \prod_{s\in \mathcal P_\sigma}B_{\sigma^i},
\]
where
\begin{equation}\label{eq: j}
j_\sigma(\alpha)=\prod_{s\in \mathcal P_\sigma} h_\sigma^*(s)\prod_{s\in\mathcal P_\sigma}h^\sigma_*(s)
\end{equation}
and
\[
A_{\sigma^i}=\begin{cases} h_\star^{\sigma}(s),\text{ if } j(s)\neq i,\\ h_\sigma^* (s)\quad\text{ otherwise}\end{cases},\quad
B_{\sigma^i}=\begin{cases} h^\star_{\sigma^i}(s),\text{ if } j(s)\neq i,\\ h^{\sigma^i}_* (s)\quad\text{ otherwise}.\end{cases}
\]
\end{lemma}
Note that in several formulas dependence on the $\alpha$ parameter is suppressed.

\section{Standardization of Wishart Distributions}\label{appendix: B}
\subsection{The moment generating function}
This section shows that the Laplace transform \eqref{FLT Mayerhofer Wishart} can be extended to its maximal domain, which is dictated by the blow up of the right side.
Clearly, $\det(e+P(\sqrt \Sigma)(u))>0$ for $u\in C$, and therefore the right side of \eqref{FLT Mayerhofer Wishart}  is well defined
as long as  $\det(e+P(\sqrt \Sigma)(u))>0$. Since
\begin{align*}
\det(e+P(\sqrt \Sigma)(u))&=\det(P(\sqrt \Sigma)(P(\sqrt{\Sigma^{-1}})e+u))\\&=\det(\Sigma) \det(P(\sqrt{\Sigma^{-1}})e+u)=\det(\Sigma) \det(\Sigma^{-1}+u),
\end{align*}
the right side of  \eqref{FLT Mayerhofer Wishart} is a real analytic function on the domain
\[
D_\Sigma:=-\Sigma^{-1}/2+C^\circ,
\]
but blows up as the argument $u$ approaches the boundary $\partial D_\Sigma$, since then the determinant vanishes for elements in $C$ that are not regular.

The following extends the validity of \eqref{FLT Mayerhofer Wishart} to
its maximal domain $D_\Sigma$:

\begin{proposition}\label{FLT maximal}
The Laplace transform of $\Gamma(\beta,\Sigma;\Omega)$
can be extended to the domain $D_\Sigma$ and
\eqref{FLT Mayerhofer Wishart} holds for any $u\in D_\Sigma$.
\end{proposition}

For the proof of this statement we require the following concerning the analytical extension of the Laplace transform of a measure on the non-negative real line (\cite[Lemma B.2]{gmm}):
\begin{lemma}\label{lem adpta}
Let $\mu$ be a probability measure on $\mathbb R_+$, and $h$ an analytic function on $(-\infty,s_1)$, where $s_1>s_0\geq 0$ such that
\begin{equation}\label{eq A1}
\int _{\mathbb R_+}e^{sx} \mu(dx)=h(s)
\end{equation}
for $s\in (-\infty, s_0)$. Then \eqref{eq A1} also holds for $s\in (-\infty, s_1)$. 
\end{lemma}

\begin{proof}[Proof of Proposition \ref{FLT maximal}]

Let $\mu^*$ be the pushforward of $\mu=\Gamma(\beta,\Sigma;\Omega)$ under
$\xi\mapsto \tr(\Sigma^{-1}\xi)=\tr((P(\sqrt{\Sigma^{-1}} e)\xi)$. Then $\mu^*$ is a probability measure on $\mathbb R_+$ with Laplace transform
\begin{align}\label{eq: one dim}
f(t):&=\int e^{t x}\mu^*(dx)=\int e^{-\tr(-t (P(\sqrt{\Sigma^{-1}} e)  ,\xi)}\mu(d\xi)\\\nonumber
&=(1-2t)^{-\beta}e^{t(1-2t)^{-1}\tr(\Sigma^{-1}\Omega)}
\end{align}
and the right side is real analytic for $t<1/2$. Hence, by Lemma \ref{lem adpta} the left side is also
finite for $t<1/2$ and equality holds in \eqref{eq: one dim}.

Therefore, we have shown that the formula \eqref{FLT Mayerhofer Wishart} can be extended to $u=-t\Sigma^{-1}$, for any
$t<1/2$. Since $u>-\Sigma^{-1}/2$ implies $u>-t\Sigma^{-1}$ for some $t<1/2$, we have for any $u>-{\Sigma^{-1}}/2$
\[
\int e^{-\tr(u\xi)}\mu(d\xi)\leq \int e^{t \tr(\Sigma^{-1}\xi)}\mu(d\xi)=f(t)<\infty
\]
and therefore the left side of \eqref{FLT Mayerhofer Wishart} exists for any $u>-\Sigma^{-1}/2$, as does the right side. The two real analytic functions (in several variables) agree on the connected open set $D_\Sigma$, since they do agree on the open subset $C^\circ$ (\cite[(9.4.2)]{dieudonne2013foundations}).
\end{proof}
\subsection{Transformations and the Natural Exponential Family}
\begin{lemma}\label{mayx}
Let $\beta\geq 0,\Omega\in C$ and $\Sigma\in C^{\circ}$.
\begin{enumerate}
\item \label{x1} If $X\sim\Gamma(\beta,e;\Omega)$, then $Y= P(\sqrt{\Sigma}) X\sim \Gamma(\beta,\Sigma;P(\sqrt{\Sigma})\Omega)$.
Conversely, $Y\sim\Gamma (\beta,\Sigma;P(\sqrt{\Sigma})\Omega)$ implies $X=P(\sqrt{\Sigma}^{-1})Y\sim\Gamma(\beta,e;\Omega)$.
\item \label{x2} If $X\sim\mu(d\xi)\sim\Gamma(\beta,te;\Omega)$, then
for $v:=\frac{1}{2}\left(-\frac{e}{t}+e\right)=\frac{1}{2}(1-\frac{1}{t})e$ there exists a random variable $Y$ distributed as
\[
\nu(d\xi):= \frac{\exp(\tr(-v\xi))\mu(d\xi)}{\mathbb E[\exp(\tr(-vX))]} \sim \Gamma (\beta,\frac{\Omega}{t^2};e).
\]\
Conversely,  $\nu(d\xi)\sim\Gamma(\beta,e;\frac{\Omega}{t^2})$ implies
that $\frac{e^{\tr(v\xi)}\nu(d\xi)}{\int e^{\tr(v\xi)}\nu(d\xi)}\sim\Gamma(\beta,te;\Omega)$.
\end{enumerate}
\end{lemma}
\begin{proof}
Proof of \ref{x1}: Using repeatedly Lemma \ref{lem: properties} \ref{prop 1a} and \ref{prop 2a}, we get
\begin{align*}
\mathbb E[e^{-\tr(u Y)}]&=\mathbb E[e^{-tr(u P(\sqrt{\Sigma}) X)}]=\mathbb E[e^{-\tr((P(\sqrt{\Sigma}) u)X)}]\\
&=(\det(e+2P(\sqrt{\Sigma}) u))^{-\beta}e^{-\tr( (u^{-1}+P(\sqrt{\Sigma}e)^{-1} P(\sqrt{\Sigma})\Omega) },
\end{align*}
and since $P(\sqrt{\Sigma}e)=\Sigma$, we have indeed $Y\sim \Gamma(\beta,\Sigma;P(\Sigma)\Omega)$. The second part can be proved
quite similarly.

Proof of \ref{x2}:
 Note that due to Proposition \ref{FLT maximal}, $v\in D_\Sigma$, where $\Sigma=t e$ and \eqref{FLT Mayerhofer Wishart} holds for $v$. Hence a random variable $Y$
with distribution $\frac{\exp(\tr(-v\xi))\mu(d\xi)}{\mathbb E[\exp(\tr(-vX))]}$ exists. We compute its Laplace transform: First,
\[
e+2 P(\sqrt{\Sigma}(u+v)=e+2 P\sqrt{\Sigma} u-P(\sqrt{\Sigma})\Sigma^{-1}+P\sqrt{\Sigma}e=P(\sqrt{\Sigma})(2u+e),
\]
because $P(\sqrt{\Sigma})\Sigma^{-1}=2\sqrt{\Sigma}(\sqrt{\Sigma}\Sigma^{-1})-(\sqrt{\Sigma})^2 \Sigma^{-1}=2e-e=e$. Hence, by
Lemma \ref{lem: properties} \ref{prop 3a}, we have for $\Sigma=t e$,
\begin{equation}\label{det compu}
\det(e+2 P(\sqrt{\Sigma}(u+v))=t\det(2u+e).
\end{equation}
Second, we have
\begin{align*}
&\langle ((u+v)^{-1}+2te)^{-1},\Omega\rangle=\langle((u+\frac{1}{2} e(1-\frac{1}{t}))^{-1}+2te)^{-1},\Omega\rangle\\
&\quad=\langle (u+\frac{1}{2} e(1-\frac{1}{t})(e+2t(u+\frac{1}{2} e(1-\frac{1}{t})))^{-1},\Omega\rangle\\&=\langle (2u+e-\frac{e}{t})(e+2u)^{-1},\frac{\Omega}{2t}\rangle\\
&=\langle ((2u+e)-\frac{1}{t}(2u+e)+\frac{2u}{t})(e+2u)^{-1},\frac{\Omega}{2t}\rangle\\&=
\frac{t-1}{2t^2}\tr(\Omega)+\langle (u^{-1}+2e)^{-1},\Omega\rangle.
\end{align*}
The combination of the last computations for the trace, in combination with \eqref{det compu}, yields
\begin{equation}\label{eq: final super}
\int e^{-\langle u+v,\xi}\mu(d\xi)=c(\Omega,t) \left(\det(u+2e)^{-\beta} e^{-\langle (u^{-1}+2e)^{-1},\Omega\rangle}\right),
\end{equation}
with the constant $c(\Omega,t)=t^{-\beta}\exp(-\frac{t-1}{2t^2}\tr(\Omega))$. Since $\nu(d\xi):=\frac{\exp(\tr(-v\xi))\mu(d\xi)}{\mathbb E[\exp(\tr(-vX))]}$ is a probability measure, by construction,
and since the term within the brackets on the right side of \eqref{eq: final super} is the Laplace transform of $\Gamma(\beta,e;\frac{\Omega}{t^2})$, we conclude that
$c(\Omega,t)=\mathbb E[\exp(\tr(-vX))]$, and thus $\nu(d\xi)\sim \Gamma(\beta,e;\frac{\Omega}{t^2})$, as claimed. The converse statement can be proved
similarly.

\end{proof}

\bibliographystyle{amsplain}

\begin{thebibliography}{10}
\bibitem{casalisletac1}
Muriel Casalis and G\'erard Letac. 
\newblock The Lukacs-Olkin-Rubin characterization of Wishart distributions on symmetric cones. \newblock {\em The Annals of Statistics}, 24(2), 763--786, 1996.

\bibitem{casalis1994characterization}
Muriel Casalis and G\'erard Letac.
\newblock Characterization of the {J}orgensen set in generalized linear models.
\newblock {\em Test}, 3(1), 145--162, 1994.

\bibitem{CFMT}
Christa Cuchiero, Damir Filipovi{\'c}, Eberhard Mayerhofer and Josef Teichmann.
\newblock Affine processes on positive semidefinite matrices.
\newblock {\em The {A}nnals of {A}pplied {P}robability}, 21(2), 397--463, 2011.

\bibitem{CKMT}
Christa Cuchiero, Martin Keller-Ressel, Eberhard Mayerhofer, and Josef
  Teichmann.
\newblock Affine {P}rocesses on {S}ymmetric {C}ones.
\newblock {\em Journal of Theoretical Probability}, 29(2), 359--422, 2016.

\bibitem{dieudonne2013foundations}
Jean Dieudonn{\'e}.
\newblock {\em Foundations of modern analysis}.
\newblock Pure and Applied Mathematics, P. Smith and S. Eilenberg, Eds. New
  York: Academic Press 10, 1969.

\bibitem{MOPS}
Ioana Dumitriu, Alan Edelman, and Gene Shuman.
\newblock Mops: Multivariate orthogonal polynomials (symbolically).
\newblock {\em Journal of Symbolic Computation}, 42(6), 587--620, 2007.

\bibitem{Faraut}
Jacques Faraut and Adam Kor{\'a}nyi.
\newblock {\em Analysis on symmetric cones}.
\newblock Oxford: Clarendon Press, 1994.

\bibitem{Gindikin}
Simon~G Gindikin.
\newblock Invariant generalized functions in homogeneous domains.
\newblock {\em Functional {A}nalysis and its {A}pplications}, 9(1), 50--52,
  1975.

\bibitem{gmm}
Piotr Graczyk, Jacek Ma\l{}ecki, and Eberhard Mayerhofer.
\newblock A {C}haracterization of {W}ishart {P}rocesses and {W}ishart
  {D}istributions.
\newblock {\em Stochastic {P}rocesses and their {A}pplications}, 128(4), 1386--1404, 2018.

\bibitem{kaneko1993selberg}
Jyoichi Kaneko.
\newblock Selberg Integrals and Hypergeometric Functions Associated with {J}ack
  Polynomials.
\newblock {\em SIAM {J}ournal on {M}athematical {A}nalysis}, 24(4), 1086--1110,
  1993.

\bibitem{letac2011existence}
G{\'e}rard Letac and H{\'e}l{\`e}ne Massam.\newblock  The Laplace transform $(\det s)^{-p} \exp \tr(s^{-1}w)$ and the existence of non-central Wishart
distributions, Journal of Multivariate Analysis 163, 96--110, 2018.

\bibitem{macdonald1998symmetric}
Ian~G.~Macdonald.
\newblock {\em Symmetric {F}unctions and {H}all {P}olynomials}.
\newblock Oxford {U}niversity {P}ress, 1998.

\bibitem{MassamNeher}
H{\'e}l{\`e}ne Massam and Erhard Neher.
\newblock On {T}ransformations and {D}eterminants of {W}ishart {V}ariables on
  {S}ymmetric {C}ones.
\newblock {\em Journal of Theoretical Probability}, 10(4), 867--902, 1997.

\bibitem{muirhead2009aspects}
Robb~J.~Muirhead.
\newblock {\em Aspects of {M}ultivariate {S}tatistical {T}heory}, volume 197.
\newblock John {W}iley \& {S}ons, New York, 2005.

\bibitem{peddada1991proof}
Shyamal~Das Peddada and Donald St.~P.~Richards.
\newblock Proof of a {C}onjecture of {ML Eaton} on the {C}haracteristic {F}unction of
  the {W}ishart distribution.
\newblock {\em The {A}nnals of {P}robability}, 19(2), 868--874, 1991.

\bibitem{stanley1989some}
Richard~P.~Stanley.
\newblock Some {C}ombinatorial {P}roperties of {J}ack {S}ymmetric {F}unctions.
\newblock {\em Advances in {M}athematics}, 77, 76--115, 1989.

\bibitem{wishart1928generalised}
John Wishart.
\newblock The {G}eneralised {P}roduct {M}oment {D}istribution in {S}amples from a {N}ormal
  {M}ultivariate {P}opulation.
\newblock {\em Biometrika}, 20A, 32--52, 1928.

\end{thebibliography}

\end{document}